\documentclass[12pt]{amsart}
\usepackage{latexsym,color,amsmath,amsthm,amssymb,amscd,amsfonts}

\usepackage{hyperref}
\usepackage{graphicx}

\newtheorem{theorem}{Theorem}[section]

\newtheorem{proposition}[theorem]{Proposition}
\newtheorem{corollary}[theorem]{Corollary}

\newtheorem{remark}[theorem]{Remark}

\begin{document}


\title[Narrowing the gaps of the missing Blaschke-Santal\'o diagrams]{Narrowing the gaps of the missing Blaschke-Santal\'o diagrams}

\author[B. Gonz\'alez Merino]{Bernardo Gonz\'alez Merino}
\address{\'Area de Matem\'atica Aplicada, Departamento de Ingenier\'ia y Tecnolog\'ia de Computadores, Facultad de Inform\'atica, Universidad de Murcia, 30100-Murcia, Spain}\email{bgmerino@um.es}

\thanks{2020 Mathematics Subject Classification. Primary 52A20; Secondary 52A21, 52A40.
Partially supported by MICINN Project PID2022-136320NB-I00 Spain.}

\date{\today}\maketitle

\begin{abstract}
We solve several new sharp inequalities relating three quantities amongst the area, perimeter, inradius, circumradius, diameter, and minimal width of planar convex bodies. As a consequence, we narrow the missing gaps in each of the missing planar Blaschke-Santal\'o diagrams. Furthermore, we extend some of those sharp inequalities into higher dimensions, by replacing either the perimeter by the mean width or the area by the volume.
\end{abstract}
\date{\today}\maketitle

\section{Introduction}

In 1916 Blaschke \cite{Blaschke} raised the question of determining the set of triples given by the volume $V(K)$, surface area $S(K)$, and integral mean curvature $M(K)$ of every set $K$ from the set of $3$-dimensional convex bodies $\mathcal K^3$. To do so, he defined the function
\[
h:\mathcal K^3\rightarrow[0,1]^2\quad\text{where}\quad h(K):=\left(\frac{4\pi S(K)}{M(K)^2},\frac{48\pi^2V(K)}{M(K)^3}\right),
\]
and studied the set of images $h(\mathcal K^3)$ of $h$, which is known as the Blaschke diagram. Despite recent efforts (see \cite{MM} or \cite{SY}) this diagram still today remains unsolved. 

Years later, Santal\'o \cite{Santalo} proposed the study of possible triples for the geometric quantities area $A$, perimeter $p$, inradius $r$, circumradius $R$, diameter $D$, and minimum width $w$ of a $2$-dimensional convex body. As a result, he already solved the cases $(A,p,w)$, $(A,p,r)$, $(A,p,R)$, $(A,D,w)$, $(p,D,w)$, and $(R,r,D)$. Later on Hern\'andez Cifre and Segura Gomis in \cite{HC} and \cite{HCSG} concluded the quest to describe the diagrams $(R,r,w)$, $(R,D,w)$, $(r,D,w)$,  Hern\'andez Cifre completely described in \cite{HC2} the diagrams $(A, D, R)$ and $(p, D, R)$, B\"or\"oczky Jr., Hern\'andez Cifre, and Salinas described the diagrams $(A, r, R)$ and $(p, r, R)$ in \cite{BHCS}, and very recently Delyon, Henrot, and Privat completely described in \cite{DHP} the diagram $(A,r,D)$.

Those functionals have been also considered in other diagrams. For instance, it has been described the case $(R,r,D,w)$ in \cite{BrGo17} as well as an almost complete description by Ting and Keller of $(A,p,D,w)$ in \cite{TK}. Furthermore, Hern\'andez Cifre et al. described in \cite{HCPSSG} diagrams for the inradius, circumradius, and volume and surface area for $n$-dimensional centrally symmetric convex bodies. In \cite{BrGo22} the authors described detailed $(R,r,D)$ diagrams when the functionals are measured with respect to other balls different from the Euclidean ball, such as the triangle, square, pentagon, or hexagon, or a general convex body. Very recently in \cite{BrRu} the authors studied $(R,r,D)$ diagrams for other notions of diameters.

More diagrams have been recently studied when introducing other functionals. Ftouhi and Lamboley \cite{FtLa} have described the diagram for the area, perimeter, and the first Dirichlet eigenvalue, Lucardesi and Zucco \cite{LuZu} described the diagram for the area, the torsional ridigity, and the first Dirichlet eigenvalue (see also Buttazzo and Pratelli \cite{BuPr} for more information on this diagram), Ftouhi and Henrot \cite{FtHe} studied the diagram for the area, the first Dirichtel eigenvalue, and the first non-trivial Laplacian-Neumann eigenvalue, and also Ftouhi \cite{Ft} computed the diagram of the area, perimeter, and the Cheeger constant. Very recently Gastaldello, Henrot, and Lucardesi in \cite{GHL} have studied the diagram of area, perimeter, and moment of inertia.

The aim of this paper is to solve several new inequalities relating three of the six functionals proposed by Santal\'o, and to clarify what is known up to date in the matter. This is why we state and prove each new result in the corresponding section. Indeed, we devote a section to each of the six unsolved diagrams, with the exception of $(A,r,w)$ and $(p,r,w)$ which use a common one.

\section{Previous inequalities}

Let $\mathcal K^n$ be the set of \emph{$n$-dimensional convex bodies}, i.e., convex and compact sets in $\mathbb R^n$ different from points.

For every $x,y\in\mathbb R^n$, $x=(x_1,x_2,\dots,x_n)$, $y=(y_1,y_2,\dots,y_n)$, let $\langle x,y\rangle:=x_1y_1+x_2y_2+\cdots+x_ny_n$ be the \emph{scalar product} of $x$ and $y$. Moreover, let $\|x\|:=\sqrt{\langle x,x\rangle}$ be the \emph{Euclidean norm} of $x$. Furthermore, let $e_1,e_2,\dots,e_n$ be the vectors of the \emph{canonical base} of $\mathbb R^n$.

Let $K\in\mathcal K^n$. The \emph{circumradius} (resp. \emph{inradius}) $R(K)$ (resp. $r(K)$) of $K$ is the smallest (resp. largest) radius of an Euclidean ball containing $K$ (resp. contained in $K$). The \emph{diameter} $D(K)$ of $K$ is the largest Euclidean length between two points contained in $K$. The \emph{(minimum) width} is the minimum distance between two parallel hyperplanes containing $K$ between them. For every $K\in\mathcal K^2$, let $p(K)$ be the \emph{perimeter} of $K$. For every $K\in\mathcal K^n$, let $V(K)$ be its \emph{volume} or \emph{Lebesgue measure}, and if $n=2$, realize that $V$ coincides with the usual \emph{area} $A$. 

These functionals are monotonically increasing with respect to inclusion, i.e., if $K,C\in\mathcal K^n$, then $f(K) \leq f(C)$ for $f\in\{R,r,D,w,V\}$ and if $n=2$ then $f(K) \leq f(C)$ for $f=p$ too. Moreover, if $f\in\{p,R,r,D,w\}$ then $f$ is $1$-homogeneous, i.e. $f(\lambda K)=\lambda f(K)$ and $\lambda\geq 0$. In the case of the volume, it is $n$-homogenous, i.e. $V(\lambda K)=\lambda^n V(K)$, for every $K\in\mathcal K^n$ and $\lambda\geq 0$, and thus $A$ is $2$-homogeneous.

Some remarkable examples of planar convex bodies are the \emph{Euclidean unit ball} $\mathbb B$, the \emph{equilateral triangle} of unit circumradius $\mathbb T$, the \emph{line segment} of edge length $2$ $\mathbb L$, and the \emph{Reuleaux triangle} of unit circumradius $\mathbb{RT}$. A body $K\in\mathcal K^2$ is of \emph{constant width} if $w(K)=D(K)$ (see \cite{MMO} for further details). In $n$-dimensional space $\mathbb R^n$, let $\mathbb B_n$ be the \emph{$n$-dimensional Euclidean unit ball}, thus having that $\mathbb B_2=\mathbb B$.

Since this is eminently a paper devoted to planar sets, when we refer to some result, we will directly specify its $2$-dimensional version. Moreover, for the references to most of the known inequalities we will refer to \cite{AS} as well as the references therein. However, we state here the inequalities that we need further on.


Let $K\in\mathcal K^2$. The isoperimetric inequality (see \cite{Oss}) states that
\begin{equation}\label{eq:isoperimetric}
4\pi A(K) \leq p(K)^2,
\end{equation}
with equality if and only if $K=\mathbb B$.

P\'al's inequality (see \cite{Pal}) provides a lower bound of $A$ in terms of $w$, namely
\begin{equation}\label{eq:Palineq}
 w(K)^2 \leq \sqrt{3} A(K),
\end{equation}
with equality if and only if $K=\mathbb T$. If $K$ is moreover assumed to be of constant width, then the inequality strengthens onto
\begin{equation}\label{eq:BLineq}
(\pi-\sqrt{3}) w(K)^2\leq 2 A(K),
\end{equation}
known as the Blaschke-Lebesgue inequality (see \cite{Bl15} or \cite{Le14}), with equality if and only if $K=\mathbb{RT}$.

Further trivial inequalities between the area and the circumradius and inradius are
\[
\pi r(K)^2 \leq A(K) \leq \pi R(K)^2,
\]
with equality if and only if $K=\mathbb B$.


The perimeter and the diameter relate as
\begin{equation}\label{eq:pD}
2D(K) \leq p(K) \leq \pi D(K),
\end{equation}
where the left inequality is trivial and gets equal when $K=\mathbb L$, and the right one is a consequence of Barbier \cite{Bar} and it becomes equality when $K$ is a constant width set.

The perimeter and the inradius fulfill the trivial inequality $2\pi r(K) \leq p(K)$ with equality for $K=\mathbb B$. The perimeter and the circumradius fulfill
\begin{equation}\label{eq:pR}
4R(K) \leq p(K) \leq 2\pi R(K),
\end{equation}
where the left inequality attains equality when $K=\mathbb L$, and the right inequality gets equal when $K=\mathbb B$.

The relation between the perimeter and the width is described by
\begin{equation}\label{eq:pw}
\pi w(K) \leq p(K)
\end{equation}
as a consequence of Cauchy's formula (see \cite{Schn}), with equality if and only if $K$ is of constant width.

In \cite{BHCS} it was proven the optimal lower bound for the perimeter in terms of the inradius and the circumradius, namely:
\begin{equation}\label{eq:LW_pRr}
    p(K) \geq 4\left(\sqrt{R(K)^2-r(K)^2}+r(K)\arcsin\left(\frac{r(K)}{R(K)}\right)\right).
\end{equation}
Moreover, equality holds if and only if $K$ is the convex hull of a point, its mirrored with the origin, and an euclidean ball centered at the origin.

In \cite{K} the author derived the optimal upper bound of the perimeter in terms of the diameter and the width
\begin{equation}\label{eq:UB_pDw_2}
p(K) \leq 2\sqrt{D(K)^2-w(K)^2}+2D(K)\arcsin(\frac{w(K)}{D(K)}).    
\end{equation}
Moreover, equality holds for the intersection of the Euclidean ball centered at the origin with a halfspace and its mirrored with respect to the origin. 


The quotient between the circumradius and the diameter is described by Jung's inequality \cite{Jung}
\begin{equation}\label{eq:jung}
\sqrt{3} R(K) \leq D(K),
\end{equation}
with equality if and only if $K$ contains an equilateral triangle of circumradius $R(K)$ and diameter $D(K)$.

The relation between the inradius and the width is described by 
\begin{equation}\label{eq:steinhagen}
2r(K) \leq w(K) \leq 3r(K),
\end{equation}
where the left side attains equality if and only if $K=\mathbb B$ and the right side, known as Steinhagen's inequality \cite{Steinh}, becomes equality if and only if $K=\mathbb T$.

The relation between the circumradius and the width is described by
\[
w(K) \leq 2R(K),
\]
with equality if and only if $K=\mathbb B$.

Finally, if $K$ is of constant width, it is known that $2 R(K) \leq (\sqrt{3}+1) r(K)$ with equality if and only if $K=\mathbb{RT}$. This last statement is a consequence of \eqref{eq:jung} and the fact that constant width sets attain equality in the concentricity inequalities 
\begin{equation}\label{eq:concentricities}
w(K) \leq r(K)+R(K)\leq D(K)
\end{equation}
(see \cite{BrGo17_2} for more information on these inequalities).
Moreover, precisely in the case of constant width sets, one can obtain the sharpened Steinhagen's inequality 
\begin{equation}\label{eq:SteinhagenCW}
2 w(K) \leq (3+\sqrt{3})r(K)
\end{equation}
with equality if and only if $K=\mathbb{RT}$.

Let $K\in\mathcal K^n$. Given a hyperplane $L\subset\mathbb R^n$, we say that $K'$ is the Steiner symmetrization of $K$ with respect to $L$ if for every $x\in L$, $(x+L^\bot)\cap K'$ is a segment centered around $x$, and it has the same length than $(x+L^\bot)\cap K$. It is well known that $D(K')\leq D(K)$, $R(K') \leq R(K)$, $r(K)\leq r(K')$, $V(K')=V(K)$, and if $n=2$ then $A(K')=A(K)$ and $p(K') \leq p(K)$ (see \cite{Schn}).

For every $X\subset\mathbb R^n$, let $\mathrm{aff}(X)$, $\mathrm{lin}(X)$, and $\mathrm{conv}(X)$ be the affine, linear, and convex hull of $X$, respectively. Given $x,y\in\mathbb R^n$, we let $[x,y]=\mathrm{conv}(\{x,y\})$ be the line segment with endpoints $x$ and $y$. Let $\partial K$ to be the boundary of $K$.

The circumball is characterized by some touching conditions of boundary points (see for instance \cite{BrKo}).

\begin{proposition}\label{prop:optcont}
    Let $K\in\mathcal K^n$ with $K\subset \mathbb B_n$. The following are equivalent:
    \begin{enumerate}
        \item $R(K)=1$.
        \item There exist $x_i\in\partial K\cap\partial\mathbb B_n$, $i=1,\dots,j$, $j\in\{2,\dots,n+1\}$, fulfilling the property that $0\in\mathrm{conv}(\{x_1,\dots,x_j\})$.
    \end{enumerate}
\end{proposition}

An analogous result applies to the inradius, namely, if $\mathbb B_n\subset K$, then $r(K)=1$ if and only if there exist $x_i\in\partial K\cap\partial\mathbb B_n$, $i=1,\dots,j$, $j\in\{2,\dots,n+1\}$, such that $0\in\mathrm{conv}(\{x_1,\dots,x_j\})$. In particular, the points $x_i$ coincide with outer normals of $K$ and $\mathbb B_n$, and thus, whenever $n=2$ and $j=3$, the intersection of the halfspaces determined by the orthogonal lines to each $x_i$ containing $x_i$, $i=1,2,3$, determines a triangle $T$ with the property that 
\begin{equation}\label{eq:r_triangle_T}
\mathbb B\subset K\subset T\quad\text{ and }\quad r(T)=1.
\end{equation}

\section{(A,p,D)}

Let us consider the diagram $f_{ApD}(\mathcal K^2)$ where
\[
f_{ApD}:\mathcal K^2\rightarrow[0,\infty)^2\quad\text{with}\quad f_{ApD}(K):=\left(\frac{p(K)}{D(K)},\frac{A(K)}{D(K)^2}\right)
\]
(see Figure \ref{fig:ApD}).

We expose the best known inequalities relating the three quantities $A$, $p$, and $D$ for $K\in\mathcal K^2$. 

First of all, notice that \eqref{eq:pD} induces a linear boundary of the diagram $f_{ApD}(\mathcal K^2)$. Indeed, $K$ fulfills \eqref{eq:pD} with equality if and only if $K$ is of constant width. In particular, from \eqref{eq:isoperimetric} and \eqref{eq:BLineq} we know that if $K$ is of constant width with $w(K)=2$, then $2(\pi-\sqrt{3}) \leq A(K) \leq 4\pi$, with $K=\mathbb{RT}$ and $K=\mathbb B$ attaining the lowest and highest values (see Figure \ref{fig:ApD}). 

Second, Kubota \cite{K} proved that
\begin{equation}\label{eq:ApDKubota}
8\varphi A(K) \leq p(K)(p(K)-2D(K)\cos\varphi),
\end{equation}
where $\varphi$ is defined to be solution of $2\varphi D(K)=p(K)\sin\varphi$, with equality for $K$ being a symmetric lenses, i.e., the intersection of two Euclidean balls of the same diameter. It is worth mentioning that \eqref{eq:ApDKubota} implies a strengthening of the isoperimetric inequality \eqref{eq:isoperimetric} (see Figure \ref{fig:ApD}\footnote{Here and on each other diagram, we represent valid inequalities by lines. In case of continuous lines, we refer to real boundaries of the diagram; in case of dashed lines, not. If we draw $K$ on a diagram, it coincides with the value $f(K)$ of its image through the corresponding diagram.}).

Third, Kubota \cite{K,K2} also proved that if $2D(K)\leq p(K)\leq 3D(K)$ then
\begin{equation}\label{eq:ApDKubota2}
4A(K)\geq (p(K)-2D(K))\sqrt{4p(K)D(K)-p(K)^2}
\end{equation}
with equality for isosceles triangles with two equal longer edges, having as extreme cases $K=\mathbb L$ and $K=\mathbb T$. Moreover, in the case of $3D(K) \leq p(K)\leq \pi D(K)$, he also proved the inequality 
\begin{equation}\label{eq:ApDKubota3}
4A(K) \geq \sqrt{3} D(K)(p(K)-2D(K)),
\end{equation}
with equality only when $K=\mathbb T$. 

\begin{figure}
    \centering
    \includegraphics[width=10cm]{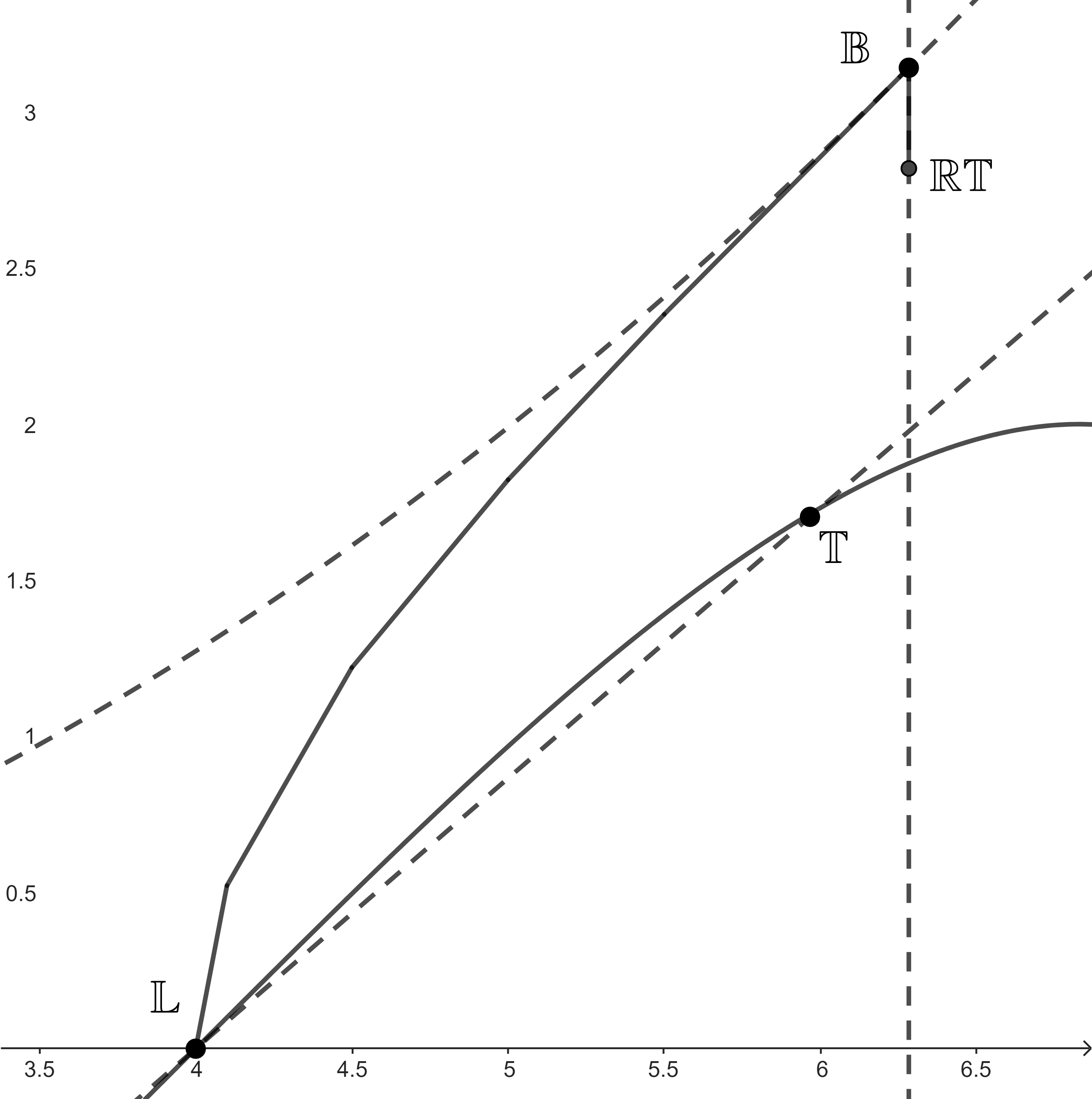}
    \caption{What we know about the diagram $(A,p,D)$, where $x=p/D$ and $y=A/D^2$. \eqref{eq:ApDKubota} provides the upper boundary, \eqref{eq:ApDKubota2} part of the lower boundary, the right-hand side of \eqref{eq:pD} gives part of the right boundary, and \eqref{eq:pD} itself and \eqref{eq:ApDKubota3} provide bounds on the remaining diagram. On top we display \eqref{eq:isoperimetric}.}
    \label{fig:ApD}
\end{figure}

\section{(p,D,r)}

Let us consider the diagram $f_{pDr}(\mathcal K^2)$ where
\[
f_{pDr}:\mathcal K^2\rightarrow[0,\infty)^2\quad\text{with}\quad f_{pDr}(K):=\left(\frac{A(K)}{D(K)^2},\frac{p(K)}{D(K)}\right)
\]
(see Figure \ref{fig:pDr_diagram}).

It is well known (see \cite[(3)]{HT}) that 
\begin{equation}\label{eq:pDrHenk}
p(K)\leq 2D(K)+4r(K),
\end{equation}
with equality for $K=\mathbb L$. 

It is also known that the right-hand side of \eqref{eq:pD} induces part of the upper boundary of the diagram, where constant width sets attain equality, and with extreme points at $K=\mathbb{RT}$ (due to \eqref{eq:SteinhagenCW}) and $K=\mathbb B$.

We now establish a new inequality, providing the lower boundary of this diagram.

\begin{figure}
    \centering
    \includegraphics[width=7cm]{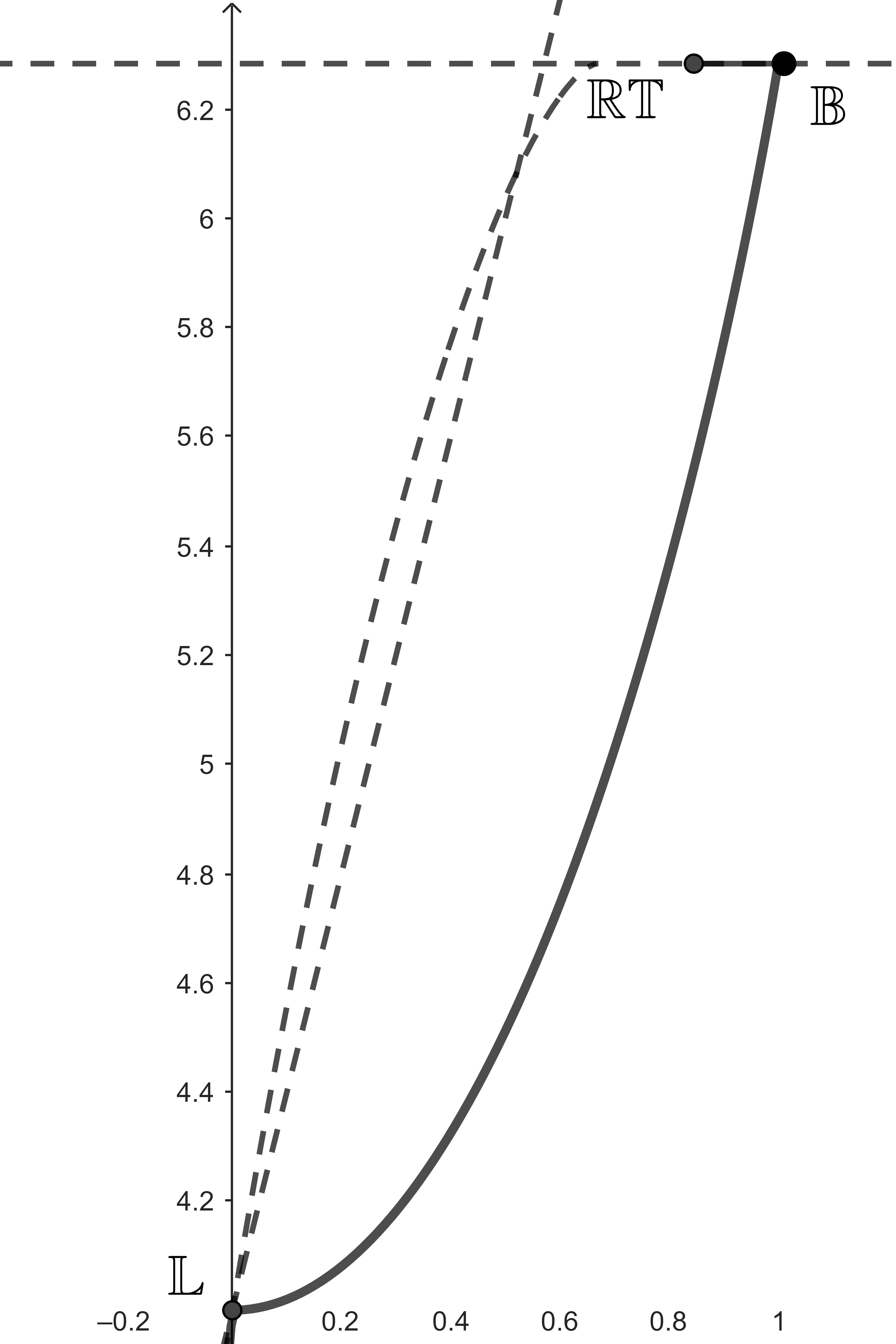}
    \caption{What we know about the diagram $(p,D,r)$, where $x=r/D$ and $y=p/D$. \eqref{eq:pD} right provides part of the upper boundary, \eqref{eq:LB(prD)} provides the lower boundary, and \eqref{eq:pD}, \eqref{eq:UB_pDr_2}, and \eqref{eq:pDrHenk} imply upper bounds.}
    \label{fig:pDr_diagram}
\end{figure}

\begin{theorem}
    Let $K\in\mathcal K^2$. Then
    \begin{equation}\label{eq:LB(prD)}
    \frac{p(K)}{r(K)} \geq 4\left(\frac{\pi}{2}-\arctan\left(\sqrt{\frac{D(K)^2}{4r(K)^2}-1}\right)+\sqrt{\frac{D(K)^2}{4r(K)^2}-1}\right).
    \end{equation}
    Moreover, for every $r\in[0,1]$, there exists $K_r\in\mathcal K^2$ for which $r(K_r)=r$, $D(K_r)=2$, and $K_r$ fulfills \eqref{eq:LB(prD)} with equality. 
\end{theorem}

\begin{proof}
    After a translation and rescaling of $K$, let us assume that $[(-1,0),(0,1)] \subset K$ with $D(K)=2$. Moreover, let $P+r(K)\mathbb B\subset K$, for some $P\in K$. 
    
    Notice that $[(-1,0),(0,1)]\cap(P+r(K)\mathbb B) \neq\emptyset$. Indeed, assume that $r(K)<1$ (as otherwise $K=\mathbb B$ and the result is true). since $P+r(K)\mathbb B\subset K\subset\{(x,y)\in\mathbb R^2:-1\leq x\leq 1\}$, if we would have that $[(-1,0),(0,1)]\cap(P+r(K)\mathbb B)=\emptyset$, then by the convexity of $K$ we can easily find an Euclidean ball of strictly larger radius than $r(K)$ within  $\mathrm{conv}([(-1,0),(0,1)]\cup(P+r(K)\mathbb B)) \subset K$, thus contradicting the definition of $r(K)$.
    
    By convexity, we have that $K_1:=\mathrm{conv}([(-1,0),(1,0)]\cup(P+r(K)\mathbb B)) \subset K$, and thus by monotonicity of $p$ then $p(K_1)\leq p(K)$. Notice that $D(K_1)=D(K)=2$ and $r(K_1)=r(K)$.

    Moreover, since $[(-1,0),(0,1)]\cap(P+r(K)\mathbb B) \neq\emptyset$, then $P+r(K)(\pm e_2)\in\partial(K_1)$, and thus the lines $P+r(K)(\pm e_2)+\mathrm{lin}(\{e_1\})$ support $K_1$. Let $K_2$ be the Steiner symmetrization of $K_1$ with respect to the line $\mathrm{lin}(\{e_1\})$. It is clear that 
    \[
    [(-1,0),(1,0)],(P_1,0)+r(K)\mathbb B\subset K_2,\quad (P_1,0)+r(K)(\pm e_2)\in\partial K_2,
    \]
    where $P=(P_1,P_2)$. Since $2=D(K_1)=D([(-1,0),(1,0)])\leq D(K_2)\leq D(K_1)$, then $D(K_2)=2$. Moreover, it is also clear that every vertical line intersecting $K_1$ is a line segment of length at most $2r(K_1)$, and so $r(K_2)\leq r(K_1)=r(K)$. This, together with $(P_1,0)+r(K)\mathbb B\subset K_2$ implies that $r(K_2)=r(K)=r(K_1)$. By properties of the Steiner symmetrization we know that $p(K_2)\leq p(K_1)$.

    Now again, we do another Steiner symmetrization $K_3$ of $K_2$, this time with respect to the line $\mathrm{lin}(e_2)$. We still have all previous properties inherited from $K_1$ on $K_2$. It is easy and almost analogous to check that $D(K_3)=D([(-1,0),(1,0)])=2$, $r(K_3)=r(K_2)=r(K)$, and that $p(K_3)\leq p(K_2)\leq p(K_1)\leq p(K)$. Finally, letting $K_4:=\mathrm{conv}([(-1,0),(1,0)]\cup(r(K)\mathbb B)) \subset K_3$, we can clearly see that
    \[
    r(K_4)=r(K),\quad D(K_4)=D(K)=2,\quad\text{and}\quad p(K) \geq p(K_4).
    \]
    Let us now compute $p(K_4)$, for given values $D(K_4)=2$ and $r(K_4)=r\in[0,1]$. To do so, let us compute the only line with negative slope supporting $r\mathbb B$ and passing through $(1,0)$. Let $y=m(x-1)$, $m<0$, be this line, such that it is tangent to $r\mathbb B$, i.e. such that
    \[
    \left\{\begin{array}{c}
         y=m(x-1) \\
         x^2+y^2=r^2 
    \end{array}\right.
    \]
    has a unique solution. Substituting $y$ onto the second equations gives $x^2+m^2(x-1)^2=r^2$, i.e. $(1+m^2)x^2-2m^2x+m^2-r^2=0$. This equation has a unique solution - root - if and only if $m^4-(1+m^2)(m^2-r^2)=0$, i.e. if and only if $m=-\frac{r}{\sqrt{1-r^2}}$. In this case, the intersecting point between the line and the circumpherence has coordinates
    \[
    x_0=\frac{m^2}{1+m^2} \quad\text{and}\quad y_0=\frac{-m}{1+m^2}.
    \]
    Since $p(K_4)$ is symmetric with respect to $\mathrm{lin}(e_i)$, $i=1,2$, its perimeter equals four times its perimeter on the first orthant. Since the angle of $(x_0,y_0)$ with the horizontal line equals $\arctan(y_0/x_0)$, we can conclude that
    \[
    \begin{split}
    p(K) \geq p(K_4) & = 4\left(r\left(\frac{\pi}{2}-\arctan(y_0/x_0)\right)+\sqrt{(x_0-1)^2+y_0^2}\right) \\
    & = 4\left(r\left(\frac{\pi}{2}-\arctan\left(\frac{\sqrt{1-r^2}}{r}\right)\right)+\sqrt{1-r^2}\right).
    \end{split}
    \]
    
    In order to derive \eqref{eq:LB(prD)}, we simply apply the above inequality to $2K/D(K)$ (which has diameter $2$) and use the $1$-homogeneity of $p$ and $r$. 

    Moreover, it is clear from the construction that every set $K=\mathrm{conv}([-e_1,e_1]\cup(r\mathbb B))$, for every $r\in[0,1]$, attains equality in \eqref{eq:LB(prD)}.
\end{proof}

We now derive an easy non tight upper bound of the diagram as a consequence of \eqref{eq:UB_pDw_2}.

\begin{corollary}
    Let $K\in\mathcal K^2$. Then
    \begin{equation}\label{eq:UB_pDr_2}
        p(K) \leq \sqrt{4D(K)^2-9r(K)^2} +2D(K)\arcsin\left(\frac{3r(K)}{D(K)}\right).
    \end{equation}
\end{corollary}

\begin{proof}
    After a suitable dilatation, let us assume that $D(K)=2$, $w:=w(K)$, and $r:=r(K)$. Using \eqref{eq:UB_pDw_2} then
    \[
    p(K) \leq 2\sqrt{4-w^2}+4\arcsin\left(\frac{w}{2}\right)=:f(w).
    \]
    Since $f(w)$ is clearly increasing on $w\in[0,2]$, using \eqref{eq:steinhagen} we can derive that
    \[
    p(K) \leq 2\sqrt{4-9r^2}+4\arcsin\left(\frac{3r}{2}\right).
    \]
    In order to obtain the inequality, we simply need to replace $K$ by $2K/D(K)$ and use the $1$-homogeneity of $p$ and $r$.
\end{proof}

\section{(p,R,w)}

Let us consider the diagram $f_{pRw}(\mathcal K^2)$ where
\[
f_{pRw}:\mathcal K^2\rightarrow[0,\infty)^2\quad\text{with}\quad f_{pRw}(K):=\left(\frac{w(K)}{R(K)},\frac{p(K)}{R(K)}\right)
\]
(see Figure \ref{fig:pRw_diagram}).

The inequality \eqref{eq:pw} induces part of the right boundary of the diagram. Remember that $K$ belongs to this boundary if and only if $K$ is of constant width. One of the extreme points is $K=\mathbb B$. On the other hand, if $K$ is of constant width, i.e. $D(K)=w(K)$, together with \eqref{eq:jung}, implies that $\sqrt{3} R(K) \leq w(K)$, with equality when $K=\mathbb{RT}$. Thus, this last set is the other extreme point of this boundary.

\begin{figure}
    \centering
    \includegraphics[width=7cm]{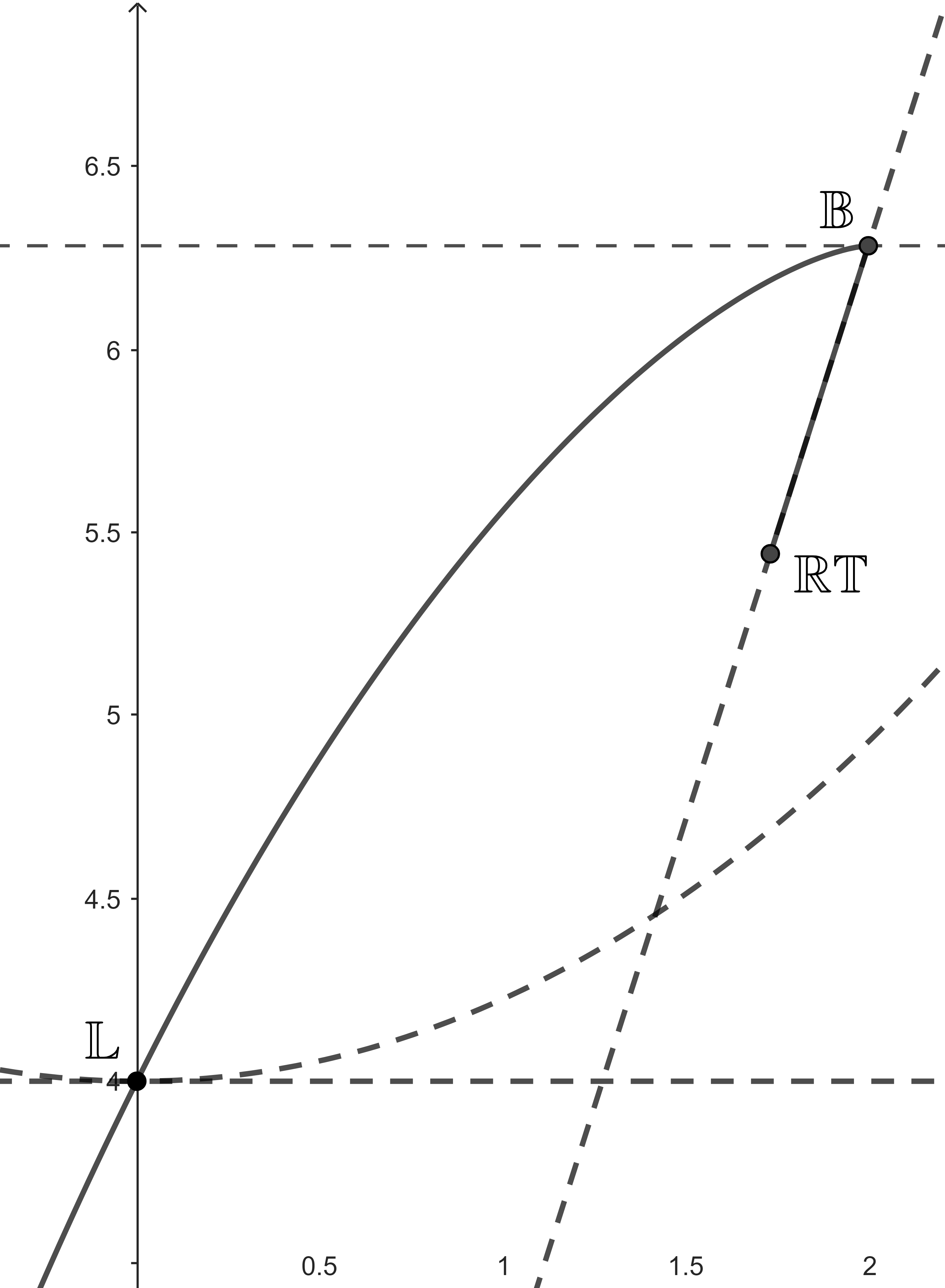}
    \caption{What we know about the diagram $(p,R,w)$, where $x=w/R$ and $y=p/R$. \eqref{eq:pw} induces part of the right boundary, \eqref{eq:pRw_NEW} induces the upper boundary, and \eqref{eq:pw} and \eqref{eq:LB(pRw)} provide bounds of the diagram. On top and bottom we show \eqref{eq:pR}.}
    \label{fig:pRw_diagram}
\end{figure}

We now show a new inequality which induces the upper boundary of the diagram.

\begin{theorem}
    Let $K\in\mathcal K^2$. Then
    \begin{equation}\label{eq:pRw_NEW}
    \frac{p(K)}{2R(K)} \leq \pi + 2 \sqrt{1-\frac{w(K)^2}{4R(K)^2}} - 2 \arccos\left(\frac{w(K)}{2R(K)}\right).
    \end{equation}
    Moreover, for every $w\in[0,2]$, there exists $K_w\in\mathcal K^2$ such that $R(K_w)=1$, $w(K_w)=w$, and $p(K_w)$ fulfills \eqref{eq:pRw_NEW} with equality.
\end{theorem}

\begin{proof}
After a suitable dilatation and rigid motion, we can assume that $R(K)=1$ with $K\subset\mathbb B$, and such that the lines $y=-a$ and $y=w-a$ support $K$, where $0\leq a\leq w-a\leq 1$ with $w:=w(K)$, i.e.
\[
\text{if }w\in[0,1]\text{ then }a\in[0,w/2]\quad\text{and if }w\in[1,2]\text{ then }a\in[w-1,w/2].
\]
Above, we take the lines containing $0$ between them. This is due to Proposition \ref{prop:optcont}, since a consequence of $R(K)=1$ with $K\subset\mathbb B$ is that $0\in K$. In particular $K\subset \mathbb B\cap\{(x,y)\in\mathbb R^2:-a\leq y\leq w-a\}=:K'$, and thus $p(K)\leq p(K')$. If we denote by $l_1$ the arc of $\mathbb S^1$ in the region $y\geq w-a$, $l_2$ the arc of $\partial\mathbb B$ in the region $y\leq -a$, $l_3:=\partial \mathbb B\setminus(l_1\cup l_2)$, $l_4$ the line-segment of $y=w-a$ within $\mathbb B$, and $l_5$ the line-segment of $y=-a$ within $\mathbb B$, then it is clear that
\[
\begin{split}
p(K') & =l_3+l_4+l_5 \\
& =2\pi-2\arccos(w-a)-2\arccos(a)+\sqrt{1-(w-a)^2}+\sqrt{1-a^2}=:f_w(a).
\end{split}
\]
Notice that
\[
f'_w(a)=2\left(\frac{w-a-1}{\sqrt{1-(w-a)^2}}-\frac{1-a}{\sqrt{1-a^2}}\right).
\]
We aim to show that the above expression is nonnegative whenever $0\leq a\leq w-a\leq 1$ (which would imply that $f_w(a)$ is nondecreasing on $a$). To do so, notice that it would then be enough to show that the function $g(x):=\frac{1-x}{\sqrt{1-x^2}}$ is decreasing for every $x\in[0,1]$. Since
\[
g'(x)=\frac{x-1}{2(1-x^2)\sqrt{1-x^2}} \leq 0
\]
for every $x\in[0,1]$, the assertion holds. Therefore
\[
p(K')=f_w(a) \leq f_w\left(\frac w2\right)=2\left(\pi-2\arccos\left(\frac{w}{2}\right)+2\sqrt{1-\left(\frac{w}{2}\right)^2}\right).
\]
In order to conclude the proof of the inequality, we simply replace $K$ by $K/R(K)$ and use the homogeneity of the functionals $p$ and $w$.

Evidently, equality holds for every $K=\mathbb B\cap\{(x,y)\in\mathbb R^2:|y|\leq w/2\}$, for every $w\in[0,2]$.
\end{proof}

We can easily derived lower bounds, for instance, using \eqref{eq:steinhagen} in combination with \eqref{eq:LW_pRr}.

\begin{corollary}
    Let $K\in\mathcal K^2$. Then
    \begin{equation}\label{eq:LB(pRw)}
    p(K) \geq 4\left(\sqrt{R(K)^2-\frac{w(K)^2}{9}}+\frac{w(K)}{3}\arcsin\left(\frac{w(K)}{3R(K)}\right)\right).
    \end{equation}
    Moreover, equality holds if $K=\mathbb L$.
\end{corollary}

\begin{proof}
    Let us suppose that $R(K)=1$, $r:=r(K)$, and $w:=w(K)$. Notice that $f(r):=\sqrt{1-r^2}+r\arcsin r$ is an increasing function in $r\in[0,1]$ since $f'(r)=\arcsin r$. Hence, using \eqref{eq:steinhagen} in combination with \eqref{eq:LW_pRr} we get that
    \[
    p(K) \geq 4\left(\sqrt{1-r^2}+r\arcsin r\right) \geq 4\left(\sqrt{1-\frac{w^2}{9}}+\frac w3\arcsin \left(\frac w3\right)\right).
    \]
    We then obtain the desired result when applying the above inequality to the set $K/R(K)$ and using the $1$-homogeneity of $p$ and $w$.

Evidently, when $K=\mathbb L$ we attain equality in \eqref{eq:LB(pRw)}.
\end{proof}

\section{(A,R,w)}

Let us consider the diagram $f_{ARw}(\mathcal K^2)$ where
\[
f_{ARw}:\mathcal K^2\rightarrow[0,\infty)^2\quad\text{with}\quad f_{ARw}(K):=\left(\frac{w(K)}{R(K)},\frac{A(K)}{R(K)^2}\right)
\]
(see Figure \ref{fig:ARw_diagram}).

Apparently, the only two known inequalities relating this triple are
\begin{equation}\label{eq:ARw_OLD}
\frac{\sqrt{3}}{2}R(K)w(K) \leq A(K) \leq 2 R(K)w(K)
\end{equation}
(see \cite{HT}) with both inequalities having equality if $K=\mathbb L$, and the left-hand side also attaining equality at $K=\mathbb T$.

\begin{figure}
    \centering
    \includegraphics[width=7cm]{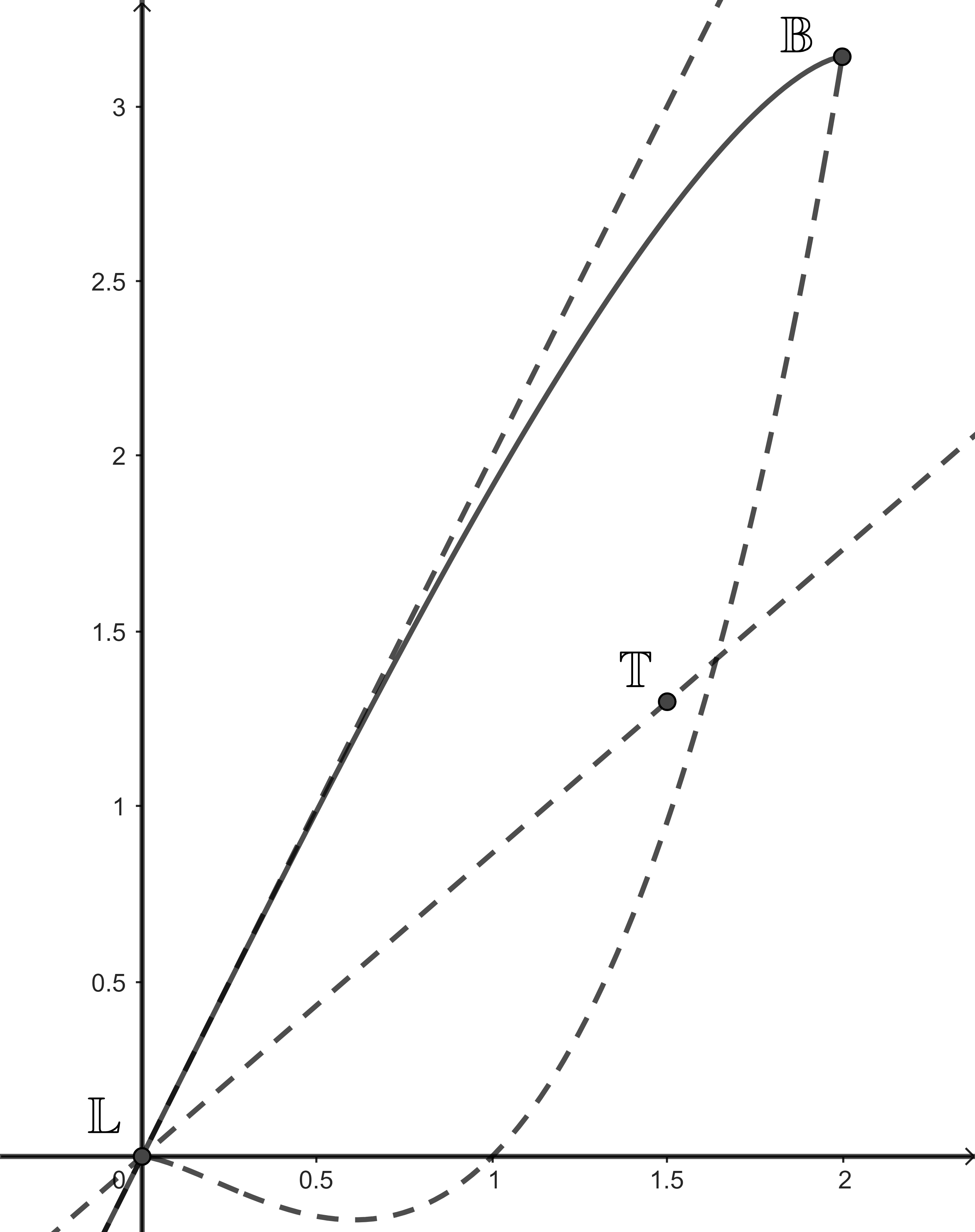}
    \caption{What we know about the $(A,R,w)$ diagram, where $x=w/R$ and $y=A/R^2$. \eqref{eq:ARw_NEW_UPPER} provides the upper boundary, \eqref{eq:ARw_LOWER} and the left-hand side of \eqref{eq:ARw_OLD} provide lower bounds for the diagram. Over \eqref{eq:ARw_NEW_UPPER} we see the right-hand side of \eqref{eq:ARw_OLD}.}
    \label{fig:ARw_diagram}
\end{figure}

The first inequality we prove induces the upper boundary of the diagram.

\begin{theorem}
    Let $K\in\mathcal K^2$. Then
    \begin{equation}\label{eq:ARw_NEW_UPPER}
    A(K) \leq R(K)^2\left(\pi-2\arccos\left(\frac{w(K)}{2R(K)}\right)+\frac{w(K)}{R(K)}\sqrt{1-\left(\frac{w(K)}{2R(K)}\right)^2}\right).
    \end{equation}
    Moreover, for every $w\in[0,2]$, there exists $K_w\in\mathcal K^2$ such that $R(K_w)=1$, $w(K_w)=w$, and $A(K_w)$ fulfill \eqref{eq:ARw_NEW_UPPER} with equality.
\end{theorem}

\begin{proof}
    After a translation and dilatation, let us assume that $K\subset\mathbb B$ with $R(K)=1$. Moreover, after a suitable rotation, let us assume that the lines supporting $K$ with width $w:=w(K)$ are given by the equations $y=w-a$ and $y=-a$, for some $a\in[0,1]$ such that $a\leq w-a$, i.e. 
    \[
    \text{if }w\in[0,1]\text{ then }a\in[0,w/2]\text{ and if }w\in[1,2]\text{ then }a\in[w-1,w/2].
    \]
    Above, we take the lines containing $0$ between them. This is due to Proposition \ref{prop:optcont}, since a consequence of $R(K)=1$ with $K\subset\mathbb B$ is that $0\in K$. 
    In particular, we have that
    \[
    K \subset \mathbb B\cap\{(x,y)\in\mathbb R^2:-a\leq y\leq w-a\}=:K'.
    \]
    Thus $A(K)\leq A(K')$. Notice that $R(K')=R(K)=1$ and $w(K')=w(K)=w$. Now let us compute $A(K')$. Notice that $K'$ is the union of four triangles and two circular sectors as in Figure \ref{fig:proof_ARw}. 

    \begin{figure}
        \centering
        \includegraphics[width=7cm]{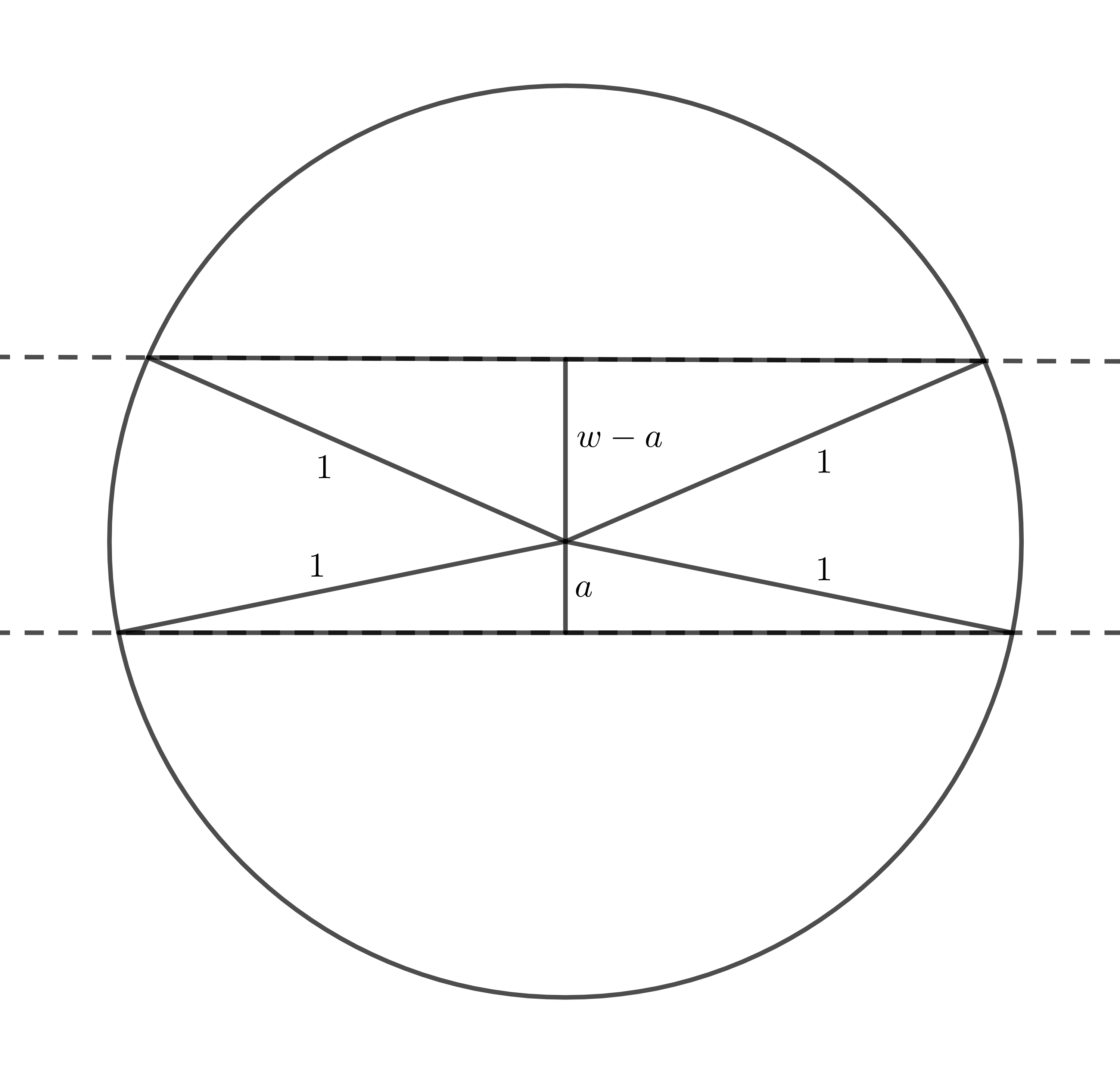}
        \caption{$K$ is contained between the lines and within $\mathbb B$.}
        \label{fig:proof_ARw}
    \end{figure}
    
    The areas of the triangles are $(w-a)\sqrt{1-(w-a)^2}/2$ and $a\sqrt{1-a^2}/2$ (twice each). The area of each of the circular sectors is $(\pi-\arccos(w-a)-\arccos(a))/2$. Thus
    \[
    \begin{split}
    A(K') & =(\pi-\arccos(w-a)-\arccos(a))+(w-a)\sqrt{1-(w-a)^2}+a\sqrt{1-a^2} \\
    & =:f_w(a).
    \end{split}
    \]
    Notice that 
    \[
    f'_w(a)=2\left(\frac{-1+(w-a)^2}{\sqrt{1-(w-a)^2}}+\frac{1-a^2}{\sqrt{1-a^2}}\right).
    \]
    We aim to show that $f'_w(a)\geq 0$ whenever $0\leq a\leq w-a\leq 1$. This is equivalent to show
    \[
    \sqrt{1-a^2}=\frac{1-a^2}{\sqrt{1-a^2}} \geq \frac{1-(w-a)^2}{\sqrt{1-(w-a)^2}}=\sqrt{1-(w-a)^2},
    \]
    for $0\leq a\leq w-a\leq 1$, i.e., that $g(x):=\sqrt{1-x^2}$ is a decreasing function. Since this last statement is true, we conclude that $f'_w(a)\geq 0$, from which
    \[
    A(K)\leq A(K')=f_w(a)\leq f_w\left(\frac w2\right)=\pi-2\arccos\left(\frac{w}{2}\right)+w\sqrt{1-\left(\frac{w}{2}\right)^2}.
    \]
    In order to retrieve \eqref{eq:ARw_NEW_UPPER}, we only need to replace above $K$ by $K/R(K)$ and use the $1$-homogeneity of $w$ as well as the $2$-homogeneity of $A$.

    Evidently, every body of the form $K=\mathbb B\cap\{(x,y)\in\mathbb R^2:|y|\leq w/2\}$ attain equality in \eqref{eq:ARw_NEW_UPPER}, for every $w\in[0,2]$.
\end{proof}

The second inequality we show relating the area, circumradius, and width is one that gives part of the right bound, partly improving \eqref{eq:ARw_OLD} for a certain range of values near $K=\mathbb B$.

\begin{theorem}
    Let $K\in\mathcal K^2$. Then
    \begin{equation}\label{eq:ARw_LOWER}
    \begin{split}
    & \frac{A(K)}{R(K)^2} \geq \\
    & \left(\pi-\arccos\left(\frac{w(K)}{R(K)}-1\right)\right)\left(\frac{w(K)}{R(K)}-1\right)^2+\left(\frac{w(K)}{R(K)}-1\right)\sqrt{2\frac{w(K)}{R(K)}-\frac{w(K)^2}{R(K)^2}}.
    \end{split}
    \end{equation}
    Moreover, equality holds if $K=\mathbb B$.
\end{theorem}

\begin{proof}
    Let us suppose that $K\subset\mathbb B$, $R(K)=1$, and $w:=w(K)$. Using the fact that $w(K)\leq R(K)+r(K)$ (see the left-hand side of \eqref{eq:concentricities}), then $r:=r(K)\geq w-1$. Let $t+r\mathbb B\subset K$ for some $t\in K$. Moreover, by Proposition \ref{prop:optcont} let $p_i\in \partial K\cap \partial\mathbb B$, $i=1,2,3$, be such that $0\in\mathrm{conv}(\{p_i:i=1,2,3\})$. Hence $\sum_{i=1,2,3}\lambda_ip_i=0$, for some $\lambda_i\geq 0$, and thus $\sum_{i=1,2,3}\lambda_i\langle p_i,t\rangle=0$ from which at least one of them, say $p_1$, fulfills $\langle p_1,t\rangle\leq 0$. Using the convexity of $K$ we get that $K'_t:=\mathrm{conv}(\{p_1\}\cup(t+r\mathbb B))\subset K$ and thus $A(K) \geq A(K'_t)$. Notice that the distance from $p_1$ to $t+r\mathbb B$ is at least $2-w$, which occurs when $t=0$. Thus $A(K'_t)\geq A(K'_0)$. It is clear that $K'_0$ decomposes in a circular sector of angle $2\pi-2\arccos(w-1)$ and radius $(w-1)$ and two triangles each of area $(w-1)\sqrt{1-(w-1)^2}/2$ (see Figure \ref{fig:proof_ARw_nonsharp}), and thus
    \[
    A(K) \geq A(K'_0) = (\pi-\arccos(w-1))(w-1)^2+(w-1)\sqrt{2w-w^2}.
    \]
    In order to obtain the result, we simply apply the above inequality to the set $K/R(K)$ and take into account the $1$-homogeneity of $w$ as well as the $2$-homogeneity of $A$.

    Moreover, if we replace $K$ by $\mathbb B$, it is very simple to notice that the inequality above becomes equality.

\end{proof}

\begin{figure}
    \centering
    \includegraphics[width=7cm]{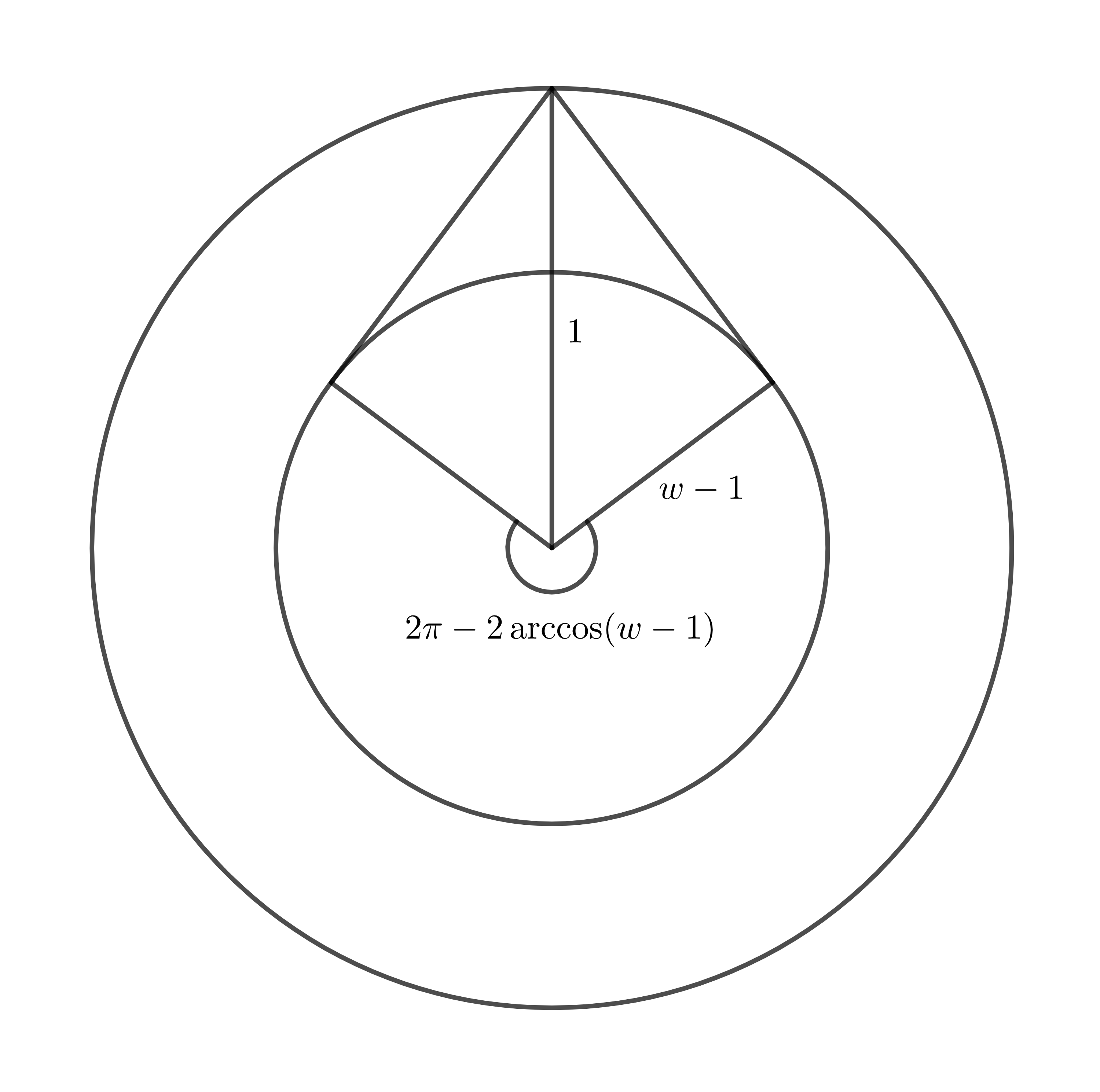}
    \caption{$K$ contains an area at least as large as this figure.}
    \label{fig:proof_ARw_nonsharp}
\end{figure}

\section{(A,r,w) and (p,r,w)}

In this last section we consider both remaining diagrams since the techniques to induce both inequalities are completely analogous.

Let us consider the diagrams $f_{Arw}(\mathcal K^2)$ as well as $f_{prw}(\mathcal K^2)$ where
\[
f_{Arw}:\mathcal K^2\rightarrow[0,\infty)^2\quad\text{with}\quad f_{Arw}(K):=\left(\frac{w(K)}{r(K)},\frac{A(K)}{r(K)^2}\right)
\]
(see Figure \ref{fig:(A,r,w)diagram}) and
\[
f_{prw}:\mathcal K^2\rightarrow[0,\infty)^2\quad\text{with}\quad f_{prw}(K):=\left(\frac{w(K)}{r(K)},\frac{p(K)}{r(K)}\right)
\]
(see Figure \ref{fig:(p,r,w)diagram}).

\subsection{$(A,r,w)$}

The left-hand side inequality of \eqref{eq:steinhagen} induces the left unbounded boundary of the diagram, with extreme point at $K=\mathbb B$. The best known classic upper bounds so far were the inequalities
\begin{equation}\label{eq:Arw_nonsharp}
4(w(K)-2r(K))A(K) \leq w(K)^3, \quad \sqrt{3}(w(K)-2r(K))A(K)\leq w(K)^2r(K)
\end{equation}
(see \cite{Scott}), where the second one attains equality when $K=\mathbb T$. The best lower bound was so far the inequality \eqref{eq:Palineq}, with equality if and only if $K=\mathbb T$ (see Figure \ref{fig:(A,r,w)diagram}).

\begin{figure}
    \centering
    \includegraphics[width=7cm]{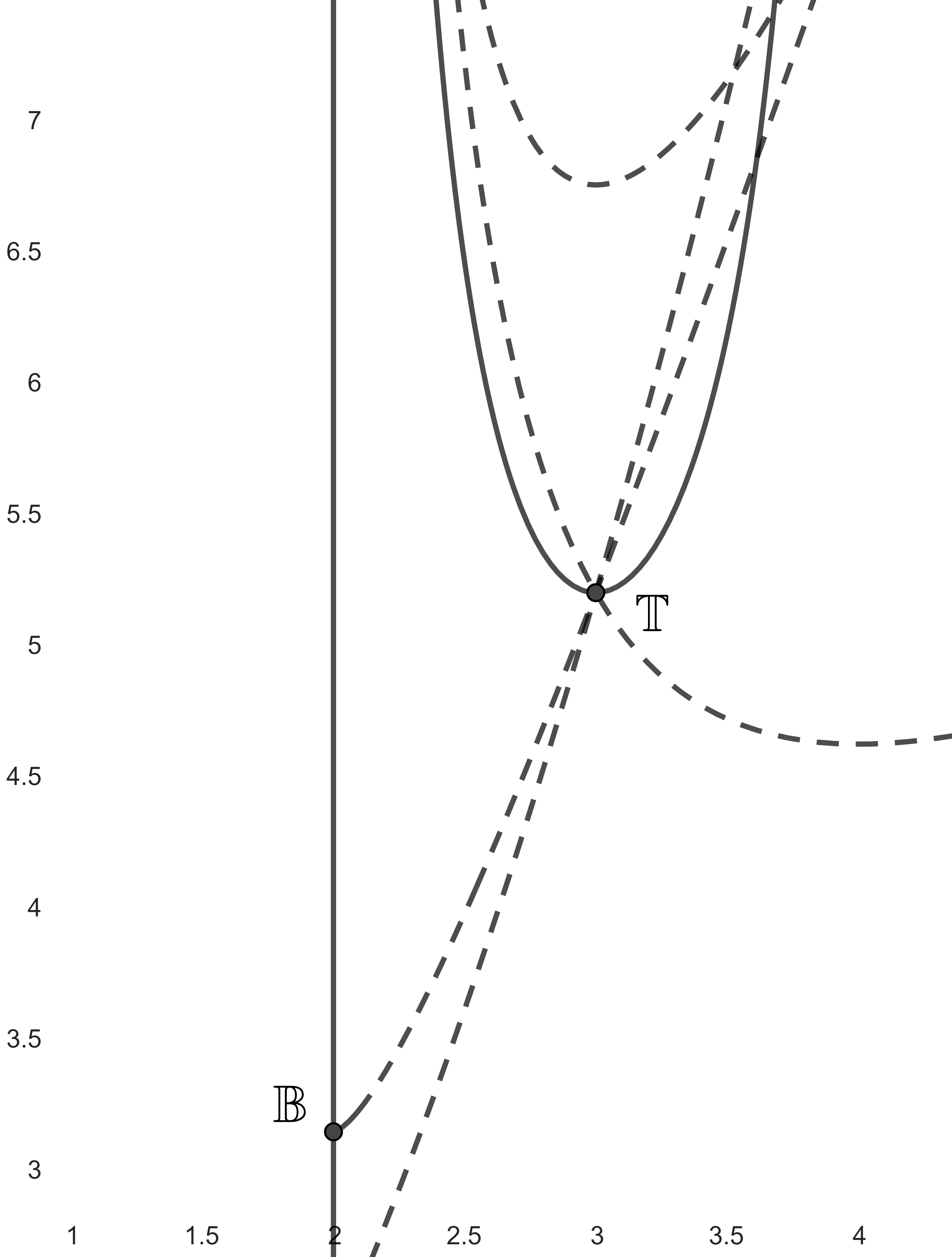}
    \caption{What we know about the diagram $(A,r,w)$, where $x=w/r$ and $y=A/r^2$. The left-hand side of \eqref{eq:steinhagen} gives the left boundary, \eqref{eq:Arw_upper} provides the upper boundary, over it we find both inequalities \eqref{eq:Arw_nonsharp}, and \eqref{eq:Arw_lower} gives the best known lower boundary of the diagram.}
    \label{fig:(A,r,w)diagram}
\end{figure}

\subsection{$(p,r,w)$}

The left-hand side inequality of \eqref{eq:steinhagen} induces the left unbounded boundary of the diagram, with extreme point at $K=\mathbb B$. Moreover, the inequality \eqref{eq:pw} provides part of the lower bound of the diagram, with extreme points at $K=\mathbb B$ and $K=\mathbb{RT}$ (due to \eqref{eq:SteinhagenCW}). The best known classic upper bound so far was the inequality 
\begin{equation}\label{eq:prw_nonsharp}
\sqrt{3}(w(K)-2r(K))p(K) \leq 2w(K)^2
\end{equation}
(see \cite{Scott}), with equality if and only if $K=\mathbb T$.

\begin{figure}
    \centering
    \includegraphics[width=7cm]{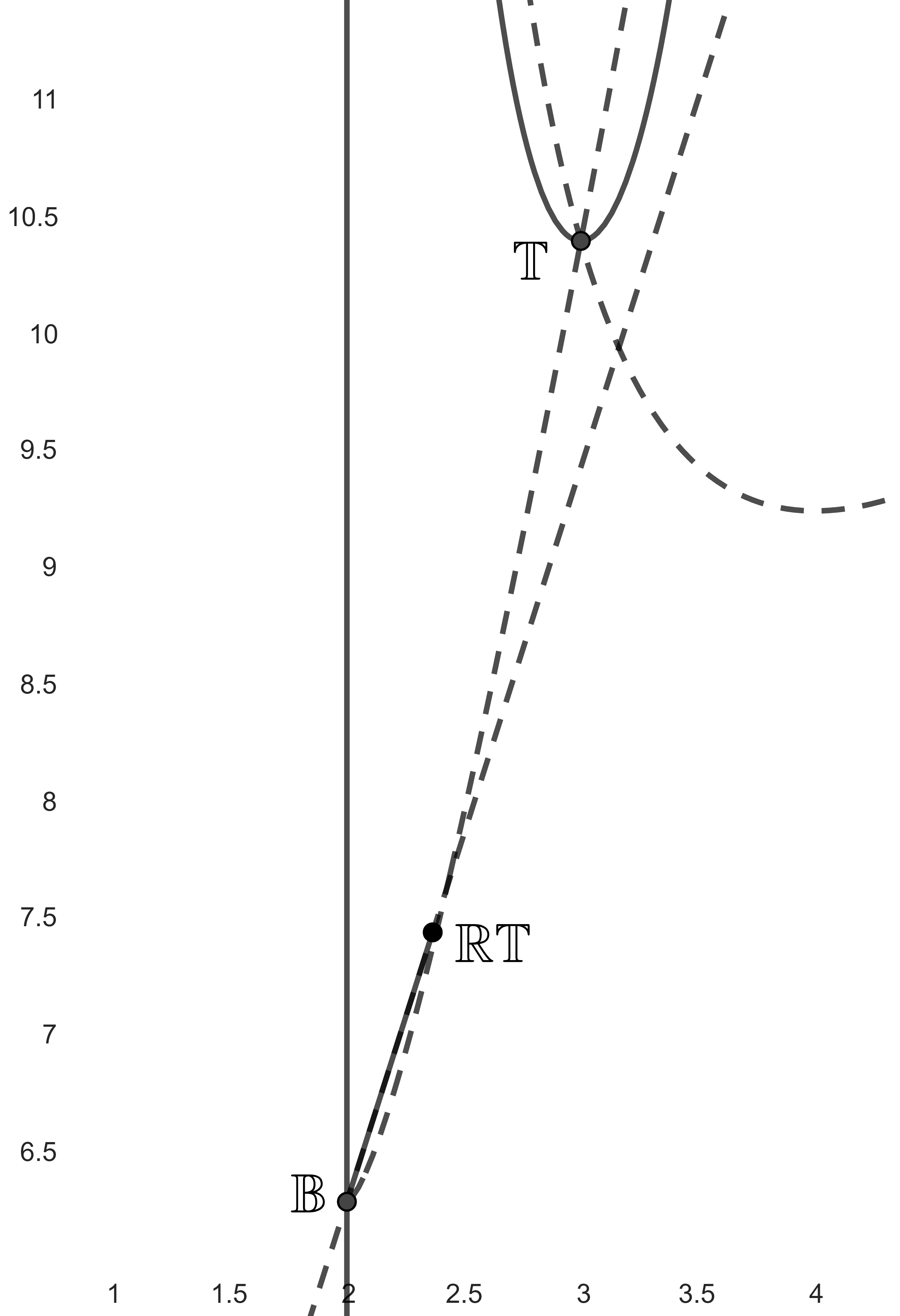}
    \caption{What we know about the diagram $(p,r,w)$. The left-hand side of \eqref{eq:steinhagen} gives the left boundary, \eqref{eq:prw_upper} gives the upper boundary, over it we find \eqref{eq:prw_nonsharp}, \eqref{eq:pw} gives part of the lower boundary, and \eqref{eq:prw_lower} provides an estimate of part of the lower boundary, improving upon \eqref{eq:pw}.}
    \label{fig:(p,r,w)diagram}
\end{figure}

The first result we show is a pair of inequalities which induce the upper boundaries of their respective diagrams.

\begin{theorem}
    Let $K\in\mathcal K^2$. Then
    \begin{equation}\label{eq:Arw_upper}
    (w(K)-2r(K))\sqrt{4r(K)-w(K)}A(K) \leq w(K)^{\frac32}r(K)^2
    \end{equation}
    and
    \begin{equation}\label{eq:prw_upper}
    (w(K)-2r(K))\sqrt{4r(K)-w(K)}p(K) \leq 2r(K)w(K)\sqrt{w(K)}.
    \end{equation}
    Moreover, for every $w\in[2,3]$ there exists an isosceles triangle $K_w$ such that $r(K_w)=1$, $w(K_w)=w$, and $A(K_w)$ (resp. $p(K_w)$) fulfills \eqref{eq:Arw_upper} (resp. \eqref{eq:prw_upper}) with equality.
\end{theorem}

\begin{proof}
    Let $T$ be a triangle such that $K\subset T$ with $r(K)=r(T)$ (see \eqref{eq:r_triangle_T}). Of course we also have that $w(K)\leq w(T)$, $A(K)\leq A(T)$, and
    $p(K)\leq p(T)$ and hence
    \[
    \frac{w(T)}{r(T)} \geq \frac{w(K)}{r(K)}, \quad 
    \frac{A(T)}{r(T)^2}\geq \frac{A(K)}{r(K)^2} \quad\text{and}\quad
    \frac{p(T)}{r(T)}\geq \frac{p(K)}{r(K)}.
    \] 
    Let us assume now that $(\pm D/2,0) \in T$, with $D:=D(T)$, and thus the width of $T$ is attained in the direction $e_2$. 

    In the next step, we will show that replacing the vertex $(x_0,w(T))$ of $T$, $x_0\in\mathbb R$, not belonging to $[(-D/2,0),(D/2,0)]$ by $(x_1,w(T))$ such that the new triangle $T'$ is an isosceles triangle with $\|(x_1,w(T))-(-D/2,0)\|=D$, has the properties $r(T')\leq r(T)$, $A(T')=A(T)$, and $p(T')\geq p(T)$. 
    To do so, let us consider the points $(-a,0),(a,0),(b,h)$, $a,b,h\geq 0$, and let 
    \[
    \begin{split}
    f_{a,h} & :=\|(b,h)-(-a,0)\|+\|(b,h)-(a,0)\| \\
    & =\sqrt{(b+a)^2+h^2}+\sqrt{(b-a)^2+h^2}.
    \end{split}
    \]
    Since
    \[
    f_{a,h}'(b)=\frac{b-a}{\sqrt{(b-a)^2+h^2}}+\frac{b+a}{\sqrt{(b+a)^2+h^2}},
    \]
    we will show that $f_{a,h}'(b)\geq 0$ for every $b\geq 0$, thus showing in particular that $p(T')\geq p(T)$. Evidently, $f_{a,h}'(b)\geq 0$ whenever $b\geq a$; thus it only remains to show it when $b\in[0,a]$. In that case, it is equivalent to show that
    \[
\sqrt{1+\frac{h^2}{(a-b)^2}} \geq \sqrt{1+\frac{h^2}{(a+b)^2}},
    \]
    which holds if and only if $(a+b)^2 \geq (a-b)^2$, which is equivalent to the valid inequality $2b\geq 0$.

    Thus, replacing $T$ by the isosceles $T'$ would give us $p(T')\geq p(T)$ (see above), and by Fubini's theorem that $A(T')=A(T)$. Hence, using the inradius, perimeter, and area formula of triangles $r(T)p(T)=2A(T)=2A(T')=r(T')p(T')$ implies that $r(T')\leq r(T)$. Evidently, we also have that $w(T')=w(T)$. Thus
    \[
    \frac{w(T')}{r(T')} \geq \frac{w(T)}{r(T)},\quad \frac{A(T')}{r(T')^2}\geq \frac{A(T)}{r(T)^2} \quad\text{and}\quad \frac{p(T')}{r(T')} \geq \frac{p(T)}{r(T)}.
    \]
    
    Wrapping up all the above, we have proven that for every $K$ there exists an isosceles triangle $T'$ with two longer edges with 
     \[
    \frac{w(T')}{r(T')} \geq \frac{w(K)}{r(K)}, \quad \frac{A(T')}{r(T')^2}\geq \frac{A(K)}{r(K)^2}  \quad\text{and}\quad \frac{p(T')}{r(T')} \geq \frac{p(K)}{r(K)}.
    \]

\begin{figure}
    \centering
    \includegraphics[width=7cm]{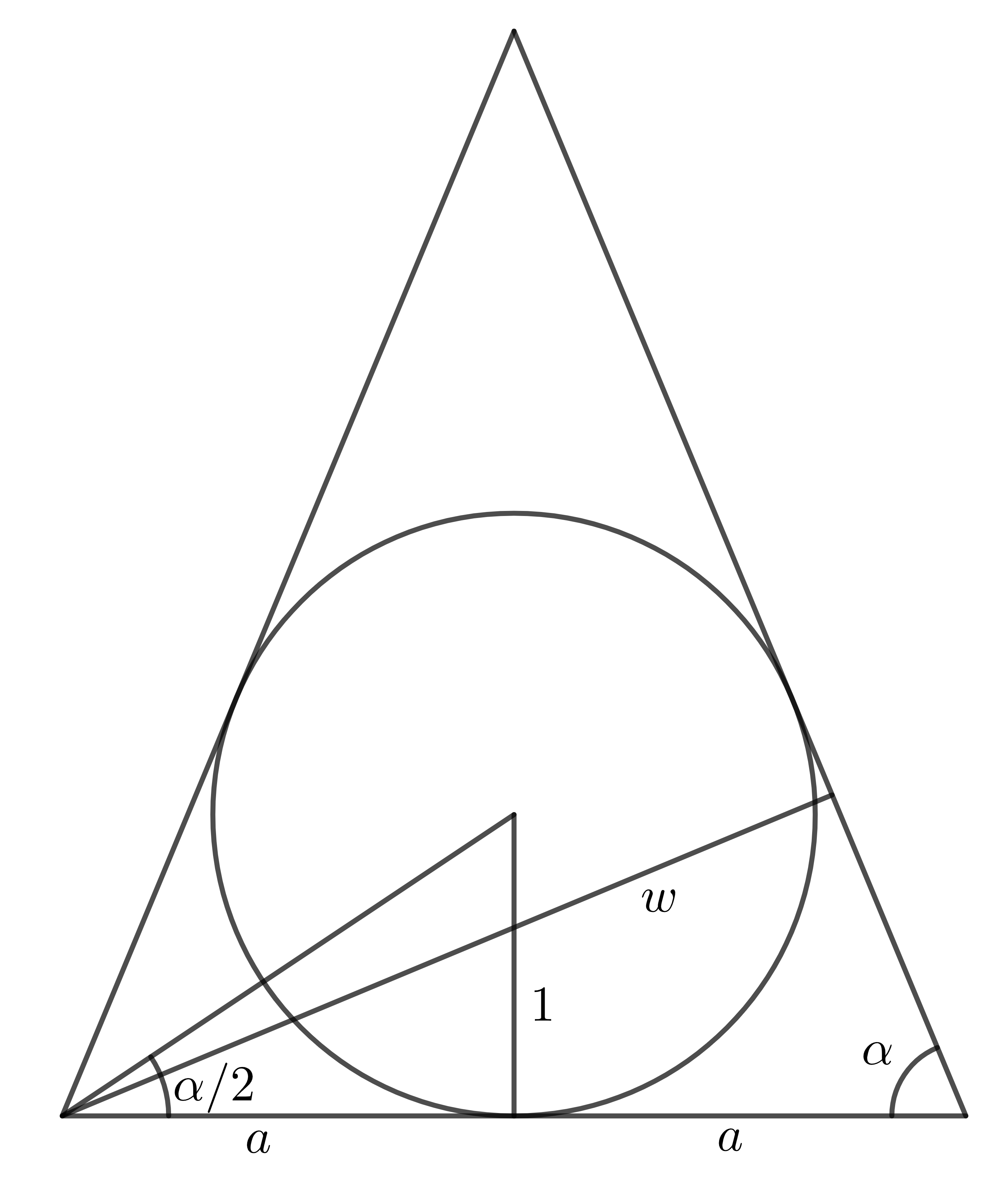}
    \caption{An isosceles triangle of one single shorter edge of length $2a$.}
    \label{fig:proof_Aprw_upper}
\end{figure}

    Finally, let $\alpha$ be one of the two equal angles of $T'$, and let $a$ be half of the length of the shorter edge of $T'$ (see Figure \ref{fig:proof_Aprw_upper}). Moreover, let us call $r:=r(T')=1$, $w:=w(T')$, $A:=A(T')$, and $p:=p(T')$. Using standard trigonometric arguments we obtain that
    \[
    \sin\alpha=\frac{w}{2a}\quad\text{and}\quad \tan\frac{\alpha}{2}=\frac1a.
    \]
    From there we easily obtain that
    \[
    a=\sqrt{\frac{w}{4-w}}\quad\text{and}\quad \alpha=2\arctan\left(\sqrt{\frac{4-w}{w}}\right)
    \]
    (notice that $w\leq 3$ by Steinhagen's theorem \eqref{eq:steinhagen}). The base of $T'$ has length $2a$ and the corresponding height equals to $h=a\tan\alpha$. 
    
    On the one hand, we have that
    \[
    \begin{split}
    A & =\frac12 (2a)h = \left(\sqrt{\frac{w}{4-w}}\right) \left(\sqrt{\frac{w}{4-w}}\tan\left(2\arctan\left(\sqrt{\frac{4-w}{w}}\right)\right)\right) \\
    & =\frac{w\sqrt{w}}{(w-2)\sqrt{4-w}} =: g_1(w).
    \end{split}
    \]
    Since $g_1'(w)=\frac{4(w-3)\sqrt{w}}{(w-2)^2(4-w)^\frac32} \leq 0$ for every $w\in[2,3]$, thus $g_1$ is non-increasing. Hence, assuming that $r(K)=r=1$ too, we conclude that
    \[
    A(K) \leq A(T')=\frac{w(T')\sqrt{w(T')}}{(w(T')-2)\sqrt{4-w(T')}} \leq \frac{w(K)\sqrt{w(K)}}{(w(K)-2)\sqrt{4-w(K)}}.
    \]
    The inequality is then retrieved by applying the inequality above to $K/r(K)$ and using the $1$-homogeneity of $w$ as well as the $2$-homogeneity of $A$.

    On the other hand
    \[
    \begin{split}
     p & = 2a+2\sqrt{a^2+h^2} \\
     & = 2\left(\sqrt{\frac{w}{4-w}}+\sqrt{\frac{w}{4-w}+\frac{w}{4-w}\tan^2\left(2\arctan\left(\sqrt{\frac{4-w}{w}}\right)\right)}\right) \\
     & =2\frac{w\sqrt{w}}{(w-2)\sqrt{4-w}}=:g_2(w).
     \end{split}
    \]
    Since
    \[
    g'_2(w)=\frac{8(w-3)\sqrt{w}}{(w-2)^2(4-w)^{\frac32}},
    \]
    thus $g_2'(w)\leq 0$ for every $w\in[2,3]$, hence getting that $g_2$ is decreasing on $w\in[2,3]$. Therefore
    \[
    \begin{split}
        p(K) \leq p(T')=2\frac{w(T')\sqrt{w(T')}}{(w(T')-2)\sqrt{4-w(T')}}
         \leq 2\frac{w(K)\sqrt{w(K)}}{(w(K)-2)\sqrt{4-w(K)}}.
        \end{split}
    \]
    In order to obtain the second result, we simply apply the inequality above to $K/r(K)$ and use the $1$-homogeneity of the functionals $p$ and $w$.

    Moreover, equality holds on each of the inequalities above for every isosceles triangle with two longer equal edges.
    \end{proof}

    We end this section by providing new inequalities that bound from below each of their diagrams. Note that \eqref{eq:Arw_lower} already provides a strenghening of P\'al's inequality \eqref{eq:Palineq}, whereas \eqref{eq:prw_lower} improves \eqref{eq:pw} within a range of values near $K=\mathbb T$.

\begin{theorem}
    Let $K\in\mathcal K^2$. Then
    \begin{equation}\label{eq:Arw_lower}
    \frac{A(K)}{r(K)^2} \geq \pi +3\left(\sqrt{\frac{w(K)^2}{r(K)^2}-\frac{2w(K)}{r(K)}}+\arcsin\left(\frac{r(K)}{w(K)-r(K)}\right)-\frac{\pi}{2}\right)
    \end{equation}
    and
    \begin{equation}\label{eq:prw_lower}
    p(K) \geq 6\left(\sqrt{w(K)^2-2w(K)}+r(K)\arcsin\left(\frac{r(K)}{w(K)-r(K)}\right)\right)-\pi r(K).
    \end{equation}
    Moreover, equality holds in both inequalities when $K=\mathbb B$ or $K=\mathbb T$.
\end{theorem}

\begin{proof}
    Let $T$ be a triangle with $K\subset T$ such that $r(T)=r(K)$ (see \eqref{eq:r_triangle_T}). Let us suppose after a translation and dilatation that $r(K)=1$ with $\mathbb B\subset K$. Let $L_i$ be the edges of $T$, $i=1,2,3$. Since $w(T)\geq w(K)=:w$, there exists $p_i\in K\notin\mathbb B$ at distance $\rho_i\geq w(K)$ from $L_i$, $i=1,2,3$. Notice also that each $p_i$ belongs to one of the three disconnected regions of $T\setminus\mathbb B$. Using the convexity of $K$, in particular we would have that \[
    K':=\bigcup_{i=1,2,3}(\mathrm{conv}(\{p_i\}\cup\mathbb B) \setminus \mathbb B)\cup\mathbb B \subset K
    \]
    with $\mathrm{int}(\mathrm{conv}(\{p_i\}\cup\mathbb B)\cap\mathrm{int}(\mathrm{conv}(\{p_j\}\cup\mathbb B)=\emptyset$, for every $i\neq j$. Thus $A(K) \geq A(K')$ and $p(K)\geq p(K')$.
    
    We first compute $A(K')$. We have that
    \[
    A(K')=A(\mathbb B)+\sum_{i=1,2,3}A(\mathrm{conv}(\{p_i\}\cup\mathbb B) \setminus \mathbb B).
    \]
    Evidently, since the distance from $p_i$ to $\mathbb B$ is at least $w-2$, we can replace $p_i$ by $p_i'$ such that the distance from $p_i'$ to $\mathbb B$ is exactly equal to $w-2$. We now compute $A(\mathrm{conv}(\{p_i\}\cup\mathbb B) \setminus \mathbb B)$. To do so, notice that the region can be seen as a rectangular triangle of hipotenuse $w-2$ and leg equal to $1$ of common angle $\alpha$ minus the corresponding ball sector from $\mathbb B$ of angle $\alpha$. On the one hand, the area of the triangle equals $\frac{\sqrt{(w-1)^2-1}}{2}$, and on the other hand the area of the ball sector equals $\frac{\alpha}{2}$, where $\alpha=\frac\pi2-\arcsin\left(\frac{1}{w-1}\right)$. Adding it all, we would conclude that
    \[
    A(K) \geq A(K')=\pi + 6\left(\frac{\sqrt{(w-1)^2-1}}{2}-\frac{\frac\pi2-\arcsin\left(\frac{1}{w-1}\right)}{2}\right).
    \]
    
    We now compute $p(K')$. Using the same ideas than above, it is clear that
    \[
    \begin{split}
    p(K) & \geq p(K')=6\sqrt{(w-1)^2-1}+2\pi\frac{2\pi-6\alpha}{2\pi} \\
     & = 6\left(\sqrt{w^2-2w}+\arcsin\left(\frac{1}{w-1}\right)\right)-\pi.
    \end{split}
    \]
    In order to obtain \eqref{eq:Arw_lower} and \eqref{eq:prw_lower}, we simply apply the inequalities above to the set $K/r(K)$ and use the $1$-homogeneity of the functionals $p$ and $w$ as well as the $2$-homogeneity of $A$.

    Moreover, if we replace $K$ by $\mathbb B$ or $\mathbb T$ then \eqref{eq:Arw_lower} and \eqref{eq:prw_lower} become equality.
\end{proof}

\section{Higher dimensions}

Some of the inequalities above can be naturally extended to higher dimensions with not much additional effort. To this aim, let $\mathbb S^{n-1}:=\partial\mathbb B_n$. Moreover, let $b(K)$ be the \emph{mean width} of $K\in\mathcal K^n$, i.e. $b(K)=\int_{\mathbb S^{n-1}}w(K;u)/2d\sigma(u)$, where $w(K;u)$ is the \emph{directional width} of $K$ in the direction $u\in\mathbb S^{n-1}$ and $\sigma$ is the \emph{unitary measure} over $\mathbb S^{n-1}$. Let $S(K)$ be the \emph{surface area measure} of $K$. Remember that $V_n(\mathbb B_n)=\pi^\frac{n}{2}/\Gamma(n/2+1)$, where $\Gamma$ stands for the \emph{Gamma function}, and that $S(\mathbb B_n)=n V(\mathbb B_n)$ for every $n\in\mathbb N$. We will stress the subindex in the volume $V_n$ when applied to some $K\in\mathcal K^n$ in case it helps in avoiding confusion.

If $L$ is an $i$-dimensional lineal subspace of $\mathbb R^n$, let $P_L$ be the \emph{orthogonal projection} onto $L$. Moreover, if $\mathrm{lin}(K)$ is $2$-dimensional, then let $p(K)$ be the perimeter of $K$ measured within $\mathrm{lin}(K)$. Notice that for every $K\in\mathcal K^n$ we have that
\begin{equation}\label{eq:b_and_p}
\begin{split}
b(K) & =\frac{1}{S(\mathbb B_n)}\int_{\mathbb S^{n-1}}\frac{w(K;u)}{2}d\sigma_{n-1}(u) \\
 & =  \frac{1}{2\pi S(\mathbb B_{n-1})}   \int_{\mathbb S^{n-2}}\int_0^{\pi} w(K;\cos\theta v+\sin\theta e_n)d\theta d\sigma_{n-2}(v) \\
 & = \frac{1}{2\pi S(\mathbb B_{n-1})} \int_{\mathbb S^{n-2}} p(P_{\mathrm{lin}(\{v,e_n\})}(K)) d\sigma_{n-2}(v).
\end{split}
\end{equation}

The first result is an extension of \eqref{eq:LB(prD)} when replacing $p(K)$ by $b(K)$.

\begin{theorem}
Let $K\in\mathcal K^n$. Then
\begin{equation}\label{eq:brD_NEW_LOWER}
    b(K) \geq \frac{2}{\pi} \left(r(K)\left(\frac\pi2-\arctan\left(\sqrt{\frac{D(K)^2}{r(K)^2}-1}\right)\right)+\sqrt{D(K)^2-r(K)^2}\right).
\end{equation}
Moreover, for every $r\in[0,1]$, there exists $K_r\in\mathcal K^n$ with $r(K_r)=r$, $D(K_r)=2$, and such that $b(K_r)$ fulfills \eqref{eq:brD_NEW_LOWER} with equality.
\end{theorem}

\begin{proof}
    After a translation and rescaling of $K$, let us assume that $[-e_1,e_1] \subset K$ with $D(K)=2$. Moreover, let $P+r(K)\mathbb B_n\subset K$, for some $P\in K$. 
    
    It is straighforward to check that $[-e_1,e_1]\cap(P+r(K)\mathbb B_n) \neq\emptyset$ (as in the proof of \eqref{eq:LB(prD)}). 
    By convexity, we have that $K_1:=\mathrm{conv}([-e_1,e_1]\cup(P+r(K)\mathbb B_n)) \subset K$, and thus by monotonicity of $b$ then $b(K_1)\leq b(K)$. Notice that $D(K_1)=D(K)=2$ and $r(K_1)=r(K)$. If we now proceed as in the proof of \eqref{eq:LB(prD)} by applying successive Steiner symmetrizations with respect to the hyperplanes $e_1,e_2,\dots,e_n$, we would easily conclude that $p(K) \geq p(K')$, where $K':=\mathrm{conv}([-e_1,e_1]\cup (r(K)\mathbb B_n))$, with $D(K')=D(K)=2$ and $r(K')=r(K)$.

    Now observe that by \eqref{eq:b_and_p} and \eqref{eq:LB(prD)} applied within $\mathrm{lin}(\{v,e_n\})$, then
    \[
    \begin{split}
    b(K) & \geq b(K') =  \frac{1}{2\pi S(\mathbb B_{n-1})}  \int_{\mathbb S^{n-2}} p(K'\cap\mathrm{lin}(\{v,e_n\}))d\sigma_{n-2}(v) \\
    & \geq \frac{2}{\pi}\Bigg[r(K'\cap\mathrm{lin}(\{v,e_n\}))\left(\frac\pi2-\arctan\left(\frac{\sqrt{1-r(K'\cap\mathrm{lin}(\{v,e_n\}))^2}}{r}\right)\right) \\
    & +\sqrt{1-r(K'\cap\mathrm{lin}(\{v,e_n\}))^2} \Bigg] \\
    & = \frac{2}{\pi} \Bigg[r(K)\left(\frac\pi2-\arctan\left(\frac{\sqrt{1-r(K)^2}}{r}\right)\right)
    +\sqrt{1-r(K)^2} \Bigg],
    \end{split}
    \]
    since $P_{\mathrm{lin}(\{v,e_n\})}(K')=K'\cap \mathrm{lin}(\{v,e_n\})$ and $r(K'\cap \mathrm{lin}(\{v,e_n\}))=r(K')$ for every $v\in\mathbb S^{n-2}$.
    
    The inequality \eqref{eq:brD_NEW_LOWER} is then retrieved by applying the above inequality to the set $K/D(K)$ and using the $1$-homogeneity of the functionals $b$ and $r$.
    
    Evidently, equality holds whenever $K=\mathrm{conv}([-e_1,e_1]\cup(r\mathbb B_n))$, for every $r\in[0,1]$.
\end{proof}

The next result is an extension of \eqref{eq:pRw_NEW} to higher dimensions, replacing $p(K)$ by $b(K)$.

\begin{theorem}
    Let $K\in\mathcal K^n$. Then
    \[
    b(K) \leq \frac{1}{\pi} R(K)\left(\pi+2\sqrt{1-\frac{w(K)^2}{4R(K)^2}}-2\arccos\left(\frac{w(K)}{2R(K)}\right)\right).
    \]
    Moreover, for every $w\in[0,2]$, there exists $K_w\in\mathcal K^n$ such that $R(K_w)=1$, $w(K_w)=w$, and $b(K_w)$ fulfills with equality the inequality above. 
\end{theorem}

\begin{proof}
After a suitable dilatation and rigid motion, we can assume that $R(K)=1$ with $K\subset\mathbb B_n$, and such that the hyperplanes $x_n=-a$ and $x_n=w-a$ support $K$, where $0\leq a\leq w-a\leq 1$ with $w:=w(K)$, i.e.
\[
\text{if }w\in[0,1]\text{ then }a\in[0,w/2]\quad\text{and if }w\in[1,2]\text{ then }a\in[w-1,w/2].
\]
In particular $K\subset \mathbb B_n\cap\{(x_1,\dots,x_n)\in\mathbb R^n:-a\leq x_n\leq w-a\}=:K'$, and thus $b(K)\leq b(K')$. Evidently, $R(K')=R(K)=1$ and $w(K')=w(K)$. 

Using \eqref{eq:b_and_p} and \eqref{eq:pRw_NEW} within $\mathrm{lin}(\{v,e_n\})$ for every $v\in\mathbb S^{n-2}$ we obtain that
\[
\begin{split}
 b(K) & \leq b(K') = \frac{1}{2\pi S(\mathbb B_{n-1})}  \int_{\mathbb S^{n-2}} p(K'\cap\mathrm{lin}(\{v,e_n\}))d\sigma_{n-2}(v)\\
& \leq \frac1\pi \left(\pi +2\sqrt{1-\frac{w(K'\cap\mathrm{lin}(\{v,e_n\}))^2}{4}}-2\arccos\left(\frac{w(K'\cap\mathrm{lin}(\{v,e_n\}))}{2}\right)\right) \\
& = \frac1\pi \left(\pi +2\sqrt{1-\frac{w(K)}{4}}-2\arccos\left(\frac{w(K)}{2}\right)\right).
\end{split}
\]
where $P_{\mathrm{lin}(\{v,e_n\})}(K')=K'\cap \mathrm{lin}(\{v,e_n\})$ and $w(K')=w(K'\cap\mathrm{lin}(\{v,e_n\})$, for every $v\in\mathbb S^{n-2}$. In order to obtain the inequality, one simply needs to replace above $K$ by $K/R(K)$ and use the $1$-homogeneity of the functionals $b$ and $w$.

Evidently, equality holds whenever $K=\mathbb B_n\cap\{(x_1,\dots,x_n)\in\mathbb R^n:|x_n|\leq w/2\}$, for every $w\in[0,2]$.
\end{proof}

Let us remember a particular instance of the \emph{Eulerian hypergeometric function}
\[
H\left(\frac12,\frac{1-n}{2};\frac32;a^2\right)=\frac1a\int_{0}^a (1-t^2)^\frac{n-1}{2}dt
\]
for every $n\in\mathbb N$ and $a\in(0,1]$ (see \cite{EMOT}).

The following result is an extension of \eqref{eq:ARw_NEW_UPPER} to higher dimensions when replacing $A(K)$ by $V(K)$.

\begin{theorem}\label{thm:VRw_new_ndim}
    Let $K\in\mathcal K^n$. Then
    \begin{equation}\label{eq:VRw_new_ndim}
    V(K) \leq \frac{\pi^\frac{n-1}{2}}{\Gamma\left(\frac{n+1}{2}\right)} w(K)R(K)^{n-1} H\left(\frac12,\frac{1-n}2;\frac32;\frac{w(K)^2}{4R(K)^2}\right).
    \end{equation}
    Moreover, for every $w\in[0,2]$, there exists $K_w\in\mathcal K^n$ such that $R(K_w)=1$, $w(K_w)=w$, and $V(K_w)$ fulfills \eqref{eq:VRw_new_ndim} with equality.
\end{theorem}

\begin{proof}
After a suitable dilatation and rigid motion, we can assume that $R(K)=1$ with $K\subset\mathbb B$, and such that the hyperplanes $x_n=-a$ and $x_n=w-a$ support $K$, where $0\leq a\leq w-a\leq 1$ with $w:=w(K)$, i.e.
\[
\text{if }w\in[0,1]\text{ then }a\in[0,w/2]\quad\text{and if }w\in[1,2]\text{ then }a\in[w-1,w/2].
\]
Above, we take the lines containing $0$ between them. This is due to Proposition \ref{prop:optcont}, since a consequence of $R(K)=1$ with $K\subset\mathbb B$ is that $0\in K$. In particular $K\subset \mathbb B\cap\{x\in\mathbb R^n:-a\leq x_n\leq w-a\}=:K'$, and thus $V(K)\leq V(K')$. 

Notice that by Fubini's principle
\[
\begin{split}
V(K') & = \int_{-a}^{w-a} V_{n-1}(K'\cap \{x\in\mathbb R^n:x_n=t\})dt \\
& = \int_{-a}^{w-a} V_{n-1}(\mathbb B_n\cap \{x\in\mathbb R^n:x_n=t\})dt = \int_{-a}^{w-a} V_{n-1}(\sqrt{1-t^2}\mathbb B_{n-1})dt\\
& = V_{n-1}(\mathbb B_{n-1})\int_{-a}^{w-a}(1-t^2)^{\frac{n-1}{2}}dt = V_{n-1}(\mathbb B_{n-1}) g_w(a).
\end{split}
\]

Since 
\[
g_w'(a)=-(1-(w-a)^2)^\frac{n-1}{2}+(1-a^2)^\frac{n-1}{2},
\]
it is very easy to show that under the conditions $0\leq a\leq w-a\leq 1$ the function $g_w'(a)$ is nonnegative. Thus 
\[
\begin{split}
V(K') & \leq V_{n-1}(\mathbb B_{n-1}) g_w\left(\frac w2\right) = V_{n-1}(\mathbb B_{n-1}) \int_{-w/2}^{w/2}(1-t^2)^\frac{n-1}{2}dt \\
& = V_{n-1}(\mathbb B_{n-1}) w H(1/2,(1-n)/2;3/2;w^2/4).
\end{split}
\]

In order to obtain \eqref{eq:VRw_new_ndim}, we simply apply the inequality above to the set $K/R(K)$ and use the $n$-homogeneity of $V_n$ and the $1$-homogeneity of $w$.

Evidently, equality holds for every $K=\mathbb B_n\cap\{x\in\mathbb R^n:|x_n|\leq w/2\}$, for every $w\in[0,2]$.
\end{proof}

The right-hand side expression on Theorem \ref{thm:VRw_new_ndim} can be explicitly computed avoiding the use of the hypergeometric function.

\begin{remark}
    If for every $n\in\mathbb N$ we denote by 
    \[
    I_n:=\int (1-t^2)^\frac{n-1}{2}dt,
    \]
    it is well-known that $I_n$ can be computed recursively in terms of $I_{n-2}$ when integrating by parts. One can derive that
    \[
    \begin{split}
    & I_2 =\frac{1}{2}(\arcsin t+t\sqrt{1-t^2})+c_2,\quad I_3=t-\frac{t^3}{3}+c_3, \\
    & I_4=\frac18(3\arcsin t+t\sqrt{1-t^2}(5-2t^2))+c_4,
    \end{split}
    \]
    for absolute constants $c_i\geq 0$, $i\geq 2$, and so on. In particular, Theorem \ref{thm:VRw_new_ndim} is equivalent to \eqref{eq:ARw_NEW_UPPER} when $n=2$; if $K\in\mathcal K^3$ we would then have that
    \[
    V(K) \leq \pi \left(w(K)R(K)^2-\frac{w(K)^3}{12}\right);
    \]
    or if $K\in\mathcal K^4$, we would obtain that
    \[
    \frac{V(K)}{R(K)^4} \leq \frac{\pi}{3}\left(3\arcsin\left(\frac{w(K)}{2R(K)}\right)+\frac{w(K)}{2R(K)}\sqrt{1-\frac{w(K)^2}{4R(K)^2}}\left(5-\frac{w(K)^2}{2R(K)^2}\right)\right).
    \]
    As mentioned in Theorem \eqref{eq:VRw_new_ndim}, sets of the form $K=\mathbb B_n\cap\{x\in\mathbb R^n:|x_n|\leq w/2\}$ would attain equality above for every $n\in\mathbb N$ and $w\in[0,2]$.
\end{remark}

\end{document}